\documentclass[10pt]{amsart}
\usepackage[all,2cell,ps]{xy}
\usepackage{amssymb}
\usepackage{amsmath}
\usepackage{amsfonts}
\usepackage{color}

\def\smsh{\wedge}
\def\latex/{{\protect\LaTeX}}
\def\latexe/{{\protect\LaTeXe}}
\def\amslatex/{{\protect\AmS-\protect\LaTeX}}
\def\tex/{{\protect\TeX}}
\def\amstex/{{\protect\AmS-\protect\TeX}}
\def\bibtex/{{Bib\protect\TeX}}
\def\makeindx/{\textit{MakeIndex}}

\usepackage{amsmath}
\usepackage{amsthm}
\usepackage{amssymb}
\usepackage{amscd}
\usepackage{amsxtra}     
\usepackage{epsfig}
\usepackage{verbatim}
\usepackage[all]{xypic}

\theoremstyle{plain} 

\newtheorem{thm}{Theorem}[section]

\newtheorem{prop}[thm]{Proposition}

\newtheorem{cor}[thm]{Corollary}
\newtheorem{defn}[thm]{Definition}
\newtheorem{rmk}[thm]{Remark}

\theoremstyle{definition}
\newtheorem{eg}[thm]{Example}

\numberwithin{equation}{section}

\numberwithin{equation}{section}

\newcommand{\CC}{\mathbb{C}}

\newcommand{\ZZ}{\mathbb{Z}}
\newcommand{\QQ}{\mathbb{Q}}
\newcommand{\Q}{\mathbb{Q}}

 \DeclareMathOperator{\Tor}{Tor}
\DeclareMathOperator{\Ext}{Ext}

\DeclareMathOperator{\len}{length}

 \DeclareMathOperator{\Supp}{Supp}
 \DeclareMathOperator{\Spec}{Spec}
 \DeclareMathOperator{\Sing}{Sing}

 \DeclareMathOperator{\Proj}{Proj}

\DeclareMathOperator{\ch}{char}
 
 \DeclareMathOperator{\pd}{pd}
 \DeclareMathOperator{\height}{height}

 \DeclareMathOperator{\coker}{coker}

\DeclareMathOperator{\ds}{\displaystyle}

\def\Tor{\textnormal{Tor}}
\def\Proj{\textnormal{Proj}}
\def\Spec{\textnormal{Spec}}
\def\alf{\otimes_{A}L}
\def\alff{\otimes_{A}L'}
\def\ki{\otimes_{L'}L}
\def\bet{\otimes_{A}E}
\def\sig{\otimes_{A}A'}
\def\ro{\otimes_{A'}L}
\def\Supp{\textnormal{Supp}}
\def\chr{\textnormal{char}}
\def\Sing{\textnormal{Sing}}
\def\H{\textnormal{H}}
\def\dm{\textnormal{dim}}
\def\htt{\textnormal{ht}}

\def\pd{\textnormal{pd}}
\def\fm{{\mathfrak m}}
\def\fn{{\mathfrak n}}
\def\len{\textnormal{length}}
\def\fp{{\mathfrak p}}

\def\ds{\displaystyle}
\def\map#1{\buildrel #1 \over \longrightarrow}

\begin{document}

\title{Hochster's  Theta Pairing and Algebraic Equivalence}
\author{Olgur Celikbas and Mark E. Walker}

\address{Department of Mathematics, University of Kansas, 645 Snow Hall, 1460 Jayhawk
Blvd, Lawrence, KS 66045-7523, USA}
\email{ocelikbas@math.ku.edu}

\address{Department of Mathematics, University of Nebraska--Lincoln, 328 Avery Hall,
Lincoln, NE 68588--0323, USA}
\email{mwalker@math.unl.edu}

\subjclass[2000]{13D03}

\keywords{Rigidity of Tor, theta pairing, hypersurface singularities}

\begin{abstract}
We define a variant of Hochster's $\theta$ pairing and prove that it is constant in flat
  families of modules over hypersurfaces with isolated
  singularities. As a consequence, we show that the $\theta$ pairing
  factors through the Grothendieck group modulo algebraic
  equivalence. Moreover, our result allows us, in certain situations,
  to translate the properties of the $\theta$ pairing in
  characteristic zero (established in  \cite{MPSW} and \cite{PV}) to the
  characteristic $p$ setting. We also give an application of our
  result to the rigidity of $\Tor$ over hypersurfaces.
\end{abstract}

\maketitle

\section{Introduction}

Let $R$ be the ring of regular functions of a hypersurface with
isolated 
singularities --- i.e., $R = S/f$ where $S$ is regular and the set
$$
\Sing(R) = \{ p \in \Spec(R) \, | \, \text{$R_p$ is not regular}\}
$$ 
is a finite
set of maximal ideals. Then for any
pair of finitely generated $R$-modules $M$ and $N$, for all $i \gg 0$, the Tor modules
$\Tor_i^R(M,N)$ have finite length and are periodic of period at most two (i.e.,
$\Tor_i^R(M,N) \cong \Tor_{i+2}^R(M,N)$).
Under these assumptions, Hochster's theta pairing \cite{Ho1} is defined as 
$$
\theta^{R}(M,N)=\len(\Tor^{R}_{2i}(M,N))-\len(\Tor^{R}_{2i+1}(M,N))
\text{ for  $i\gg 0$.}
$$
The $\theta$ pairing is additive on short exact sequences in each
argument, and thus determines a $\ZZ$-valued pairing on $G(R)$, the Grothendieck
group of finitely generated $R$-modules. One looses no information by
tensoring with $\QQ$, and so often $\theta$ is interpreted as a
symmetric bilinear form on the rational vector space $G(R)_\QQ = G(R) \otimes \QQ$.

Hochster \cite{Ho1} originally defined the theta pairing (for a larger
class of pairs of modules) to study the
direct summand conjecture. Recently the
$\theta$ pairing was examined in detail in \cite{Da1}, \cite{Da2}
and \cite{MPSW}. 
Dao \cite[2.8]{Da1} observed an important connection of the
$\theta$ pairing with the rigidity of $\Tor$: Assume $(R, \fm)=S/(f)$ where $S$ is an unramified (or
equi-characteristic) regular local ring and $R_\fp$ is regular for all
$\fp \neq \fm$. Then the vanishing of
$\theta^{R}(M,N)$ implies that the pair of finitely generated $R$-modules
$(M,N)$ is ``$\Tor$-rigid'' in the sense that if $\Tor^{R}_{n}(M,N)=0$ for
some non-negative integer $n$, then $\Tor^{R}_{i}(M,N)=0$ for all
$i\geq n$. (We will record a slightly more general version of this
fact in Proposition \ref{Dao}.)

In \cite{MPSW}, Moore et. al. study the $\theta$ pairing for 
rings of the form
$$
\displaystyle{R =  k[x_0, \dots, x_n]/f(x_0, \dots, x_n)},
$$
where $k$ is an algebraically closed field and $f$ is a homogeneous
polynomial of degree $d$  such that $X := \Proj(R)$ is a smooth $k$-variety. (So, $\fm =
(x_0, \dots, x_n)$ is the only non-smooth point of $R$.) 
They show that the $\theta$ pairing is induced, via the Chern character map, from
the pairing on the primitive part of $\ell$-adic \'etale  cohomology 
$$
\frac{H^{\frac{n-1}{2}}_{et}(X, \QQ_\ell)}{\QQ_\ell \cdot \gamma^{\frac{n-1}{2}}}
\times 
\frac{H^{\frac{n-1}{2}}_{et}(X, \QQ_\ell)}{\QQ_\ell \cdot
  \gamma^{\frac{n-1}{2}}} \to \QQ_\ell
$$
given by
$$
(a, b) \mapsto 
\left(\int_X a \cup \gamma^{\frac{n-1}{2}}\right) 
\left(\int_X b \cup \gamma^{\frac{n-1}{2}}\right) - d \int_X a \cup b.
$$
(Here, $\ell$ is any prime distinct from
$\ch(k)$ and  $\gamma$ is the class of a
hyperplane section.) 
In particular, the theta
pairing vanishes for rings $R$ of
this type having even dimension (i.e., for which $n$ is even). 
Moreover, in the odd dimensional case, when $\ch(k) = 0$, using the
Hodge-Riemann bilinear relations, they show that the above
pairing on \'etale cohomology is positive/negative definite depending
whether $n$ is congruent to $3$ or $1 \pmod{4}$.
In particular, $(-1)^{\frac{n+1}{2}} \theta$ is
semi-positive definite in this case.

Closely related to the $\theta$ pairing is the Herbrand difference
$h(M,N)$, introduced by Buchweitz \cite{Bu}, which is defined by replacing
$\Tor$ with $\Ext$ in the formula defining $\theta$. 
Under mild hypotheses, these two invariants are essentially the same: If
$R$ is finite type over a field, then 
upon tensoring with $\Q$, each pairing is induced from a pairing on
the graded (rational) Chow group, $CH^\cdot(R)_\Q$, of $R$. On
the summand $CH^j(R)_\Q$, the
$\theta$ pairing and the Herbrand
difference agree when $j$ is even, and they differ by a
sign when $j$ is odd.

When $R$ has the form $\CC[[x_0, \dots,
x_n]]/f$ with $f$ a power series with an isolated singularity, the
Herbrand difference coincides with Euler characteristic for the
category of matrix factorizations of $f$. Recently, Polishchuk and Vaintrob
\cite[Theorem 4.1.4]{PV} have established a
``Riemann-Roch'' formula of the form
$$
h(M,N) = \langle ch(M), ch(N) \rangle
$$
for maximum Cohen-Macaulay $R$-modules $M$ and $N$.
Here, $ch(M)$ is the ``Chern character'' of $M$, which is a certain
element in the Milnor algebra of $f$:
$$
\frac{\CC[[x_0, \dots, x_n]]}{\left(\frac{\partial f}{\partial x_0}, \dots,
\frac{\partial f}{\partial x_n}\right)}.
$$
The pairing $\langle -,-
\rangle$ on the Milnor algebra is given by the formula involving
generalized fractions
$$
\langle g,h \rangle 
=
Res 
\left[
gh \, dx_0 \smsh \cdots \smsh dx_n 
\over
\frac{\partial f}{\partial x_0}, \dots, 
\frac{\partial f}{\partial x_n}
\right].
$$
The class $ch(M)$ is given in terms of a matrix factorization of $f$
that presents $M$, and is defined in terms of  the ``boundary-bulk
map'', a construction arising in mathematical physics. In
particular, $h(M,N)$ may be computed directly using the calculus of
generalized fractions from the entries of the matrix factorization
representations of $M$ and $N$. Moreover,
since the boundary-bulk map vanishes when $n$ is even, this result shows
that the $h$ pairing, and hence the $\theta$ pairing,
vanish for even dimensional rings $R$ of this type.
This exciting connection 
between homological algebra of hypersurfaces and mathematical physics is only beginning
to be explored.

The main point of this article is to establish the invariance of the
$\theta$ pairing in flat families --- see Theorem \ref{main} for a
precise statement. For example, suppose $A$ is a complete dvr, $f$
belongs to $C := A[x_0,
\dots, x_n]$, and $B := C/f$ is flat over $A$ and its singular locus
is module finite over $A$. Suppose also that $M$ and $N$ are a pair of finitely generated
$B$-modules that are $A$-flat. Then given a map $A \to L$ with $L$  
an algebraically closed field, the integer $\theta^{B \otimes_A L}(M
\otimes_A L, N \otimes_A L)$ is independent of $L$ (and the map $A \to
L$). 

In particular, taking $A$ to be the $p$-adic integers, our result
allows us to translate, in certain situations, results in characteristic
$0$, such as those listed above,  to results in characteristic
$p$ --- see Corollary \ref{C3} for precise statements of this sort. In the equi-characteristic setting, if we take $A$ to be the ring
of regular functions of a
smooth affine curve  over an algebraically closed field $k$, then our
results imply that the $\theta$ pairing does not distinguish between
algebraically equivalent elements of $G(R)$. In other words, $\theta$ descends to a 
pairing on the Grothendieck group of $R$ modulo algebraic equivalence --- see
Corollary \ref{C2} for a
precise statement.

\section{Main Theorem} \label{MTsection}

Throughout this paper, all rings are assumed to be commutative and Noetherian.
We give some definitions used in our main theorem.

\begin{defn} Let $R$ be a ring. An $R$-algebra $S$ is
  {\em essentially of finite type} over $R$ if $S$ is a localization
  of a finitely generated $R$-algebra. $S$ is said to be
  {\em essentially smooth} over $R$ if (i) $S$ is essentially of finite
  type over $R$, (ii) the map $R \to S$ is flat, and (iii) for all ring maps $R \to L$ where $L$ is a field, the ring $L\otimes_{R}S$ is regular.
\end{defn}

For example, if $k$ is a field and $S$ is a smooth
$k$-algebra --- that is, 
$$
S = k[x_1, \dots,x_n]/(f_1, \dots, f_m)
$$ 
and the Jacobian matrix $\ds{\left( \frac{\partial
f_i}{\partial x_j} \right)\otimes_{S}S/\fm}$ has rank equal to
$\ds{n-\dm(S_\fm)}$ for each maximal ideal $\fm$ of $S$ --- then any localization of $S$ is essentially smooth over $k$.

\begin{defn} \label{D1}
A {\em flat family of  hypersurfaces with isolated singularities}
consists of
  a base ring $A$, an $A$-algebra $C$, and an element $f \in C$ such
  that, letting $B := C/(f)$, the following conditions hold:
\begin{enumerate}
\item $\Spec(A)$ is connected,
\item $C$ is essentially smooth over $A$,
\item $B$ is flat over $A$, and
\item there exists an ideal  $I$ of $B$ such that $B/I$ is module finite over $A$ and
$B_\fp$ is essentially smooth over $A$ for all $\fp \in \Spec(B)
  \setminus V_B(I)$.
\end{enumerate}
\end{defn}

\begin{rmk} \label{R1}
For the ideal $I$ in Definition \ref{D1}, $V(I)$  contains the singular locus of the map
$A \to B$, and so the assumption that $A \to B/I$ is module finite
implies that $A \to B$ is a family of isolated singularities. In more
detail, 
given $A \to L$ with $L$  a field, 
for $q \in \Spec(B\otimes_{A}L)$, let $p = q \cap B$. Then 
$(B\otimes_{A}L)_{q}$ is a localization of $(B\otimes_{A}L)_{p}$
and $(B\otimes_{A}L)_{p}\cong B_{p}\otimes_{A}L$.  In particular, if
$p$ is a regular point of $\Spec(B)$, then $q$ is a regular point of
$\Spec(B \otimes_A L)$. 
It follows that $\ds{\Sing(B\otimes_{A}L) \subseteq
  V_{B\otimes_{A}L}(J)}$, where $J := I \cdot (B \otimes_A L)$.
Since 
$\ds{(B\otimes_{A}L)/J}  \cong (B/I) \otimes_A L$ is a finite
dimensional $L$-algebra, $\Sing(B\otimes_{A}L)$ is a finite set of
maximal ideals.
Also notice that 
$B \otimes_A L = (C \otimes_A L)/(f)$ and $C \otimes_A L$ is regular,
by assumption, and hence
$B \otimes_A L$ is a hypersurface with isolated singularities.
\end{rmk}

\begin{eg}  Suppose $A$ is a complete dvr (e.g., the $p$-adic
  integers) and $C = A[x_0, \dots, x_n]_{(x_0, \dots,
    x_n)}$. 
Say $f \in C$ is such that
  $(\frac{\partial f}{\partial x_0}, \cdots, \frac{\partial
      f}{\partial x_n})$ (where $\frac{\partial -}{\partial x_i}$ is
    the evident $A$-linear derivation of $C$) contains $(x_0^N, \dots, x_n^N)$ for some $N
    \gg 0$. For example,
  $f = x_0^{d_0} + \cdots + x_n^{d_n}$ satisfies this condition
  provided the integers $d_0, \dots, d_n$ are invertible in $A$. Then
  $A,C$ and $f\in C$ form a flat family of hypersurfaces with isolated singularities. The required
  ideal $I$ may be taken to be
$(\frac{\partial f}{\partial x_0}, \cdots, \frac{\partial
      f}{\partial x_n})$.
\end{eg}

\begin{defn} \label{theta} For a flat family of hypersurfaces with isolated
  singularities as in Definition \ref{D1},
  if $M$ and $N$ are finitely generated $B$-modules that are $A$-flat
  and $\alpha: A \to L$ is a ring map with $L$ a field, 
we define
$$
\begin{aligned}
\widetilde{\theta}_{\alpha}(M,N) :=
\dim_{L}  & \left(\Tor_{2i}^{B\otimes_{A}L}(M\otimes_{A}L,  N\otimes_{A}L)\right)
- \\
&
\dim_{L}\left(\Tor_{2i-1}^{B\otimes_{A}L}(M\otimes_{A}L,N\otimes_{A}L)\right) 
\,  \text{ for $i \gg 0$.}
\\
\end{aligned}
$$
\end{defn}

Since $B \otimes_A L$ is a hypersurface, we have
$$
\ds{\Tor^{B\otimes_{A}L}_{i}(M\otimes_{A}L,N\otimes_{A}L)\cong
  \Tor^{B\otimes_{A}L}_{i+2}(M\otimes_{A}L,N\otimes_{A}L)}
$$ 
for all
$i\gg 0$. 
Moreover, since 
$\Sing(B\otimes_{A}L)  \subset V_{B \otimes_AL}(I \cdot (B  \otimes_A
L))$ and $(B \otimes
_A L)/I \cdot(B \otimes_A L)$ is finite dimensional over $L$ by Remark \ref{R1}, it
follows that 
$\ds{\Tor^{B\otimes_{A}L}_{i}(M\otimes_{A}L,N\otimes_{A}L)}$ 
is supported on
$V(I \cdot(B \otimes_A L))$
and is finite dimensional over $L$, for all $i\gg 0$. In
particular, $\ds{\widetilde{\theta}_{\alpha}(M,N)}$ is well-defined.

\begin{rmk} \label{rmk2}
The pairing $\widetilde{\theta}_\alpha$ is closely related to the
$\theta$ pairing. For observe that if $L$ is algebraically closed,
then the residue of every prime in the Artinian ring $(B \otimes_A
L)/I \cdot (B \otimes_A L)$ is isomorphic to $L$ (by 
Nullstellensatz). Since
$\ds{\Tor^{B\otimes_{A}L}_{i}(M\otimes_{A}L,N\otimes_{A}L)}$ is
supported on $V(I \cdot(B \otimes_A L))$ for all $i \gg 0$, it follows
that 
$$
\dim_L \Tor^{B \otimes_A L}_i(M \otimes_A L, N \otimes_A L) = 
\len_{B \otimes_A L} \Tor^{B \otimes_A L}_i(M \otimes_A L, N \otimes_A L)
$$
for all $i \gg 0$, and hence
$$
\widetilde{\theta}_{\alpha}(M,N) = \theta^{B \otimes_A L}(M
\otimes_A L, N \otimes_A L)
$$
for such $L$. Typically, for questions involving $\theta$, passage
along a faithfully flat extension looses no information, and thus
one may reduce to a situation in which the residue fields are
algebraically closed.
\end{rmk}

The main result of this paper is the following. (See  Definitions
\ref{D1} and \ref{theta} for the terminology.)

\begin{thm}\label{main} Let $A$, $C$, and $f \in C$ be a flat
  family of hypersurfaces with isolated singularities. Set $B = C/f$ and let $M$ and $N$ be
  finitely generated $B$-modules both of which are flat as
  $A$-modules. If $\alpha: A \to L$ is a ring map with $L$ a field, then 
$\ds{\widetilde{\theta}_{\alpha}(M,N)}$
is independent of $L$ and $\alpha$.
\end{thm}

\begin{eg} Let $A = \ZZ$, $\ds{C=\ZZ[x_{1}, \dots, x_{n}, y_{1}, \dots,
    y_{n}]}$, $\ds{f=\sum_{i=1}^{n}x_{i}y_{i}}$, $M=B/(x_{1}, \dots,
  x_{n})$ and $N=B/(y_{1}, \dots, y_{n})$.
Then $A,C,f$ form a flat family of 
hypersurfaces with isolated singularities (for which $I$ may be taken to be $(x_0,
\dots, x_n, y_0, \dots, y_n)$). 
Our Theorem \ref{main} implies
$\ds{\widetilde{\theta}_{\alpha}(M,N)}$ is independent of $\alpha$. 
This can be checked directly: $\Tor_{i}^{B\otimes_{A}L}(M\otimes_{A}L,N\otimes_{A}L) \cong L$ if
$i \geq 0$ is even and
$\Tor_{i}^{B\otimes_{A}L}(M\otimes_{A}L,N\otimes_{A}L)=0$ if $i$ is
odd.
(See, for example, \cite[3.12]{CeD}.) Therefore $\widetilde{\theta}_{\alpha}(M,N)=1$ for all $\alpha$.
\end{eg}

\begin{proof}[Proof of Theorem \ref{main}]
Let $E$ be a field and let $\beta: A \rightarrow E$ be a ring map; 
it suffices to prove that 
$\displaystyle{\widetilde{\theta}_{\alpha}(M,N)=\widetilde{\theta}_{\beta}(M,N)}$.
Note that, for some $p\in \Spec(A)$, we have a commutative diagram of ring maps of the form:
$$\xymatrix{
  A\ar[dr]_{\alpha} \ar[r]^{\alpha'} & k(p) \ar[d]^{}  \\
 & L
  }$$
Moreover 
$$
\Tor^{B\alff}_{i}(M\alff,N\alff)\ki \cong
\Tor^{B\alf}_{i}(M\alf,N\alf)
$$ 
where $L'=k(p)$, and
hence $\widetilde{\theta}_{\alpha}(M,N)=\widetilde{\theta}_{\alpha'}(M,N)$.
Without loss of generality we may therefore assume $L=k(p)$ and
$E=k(q)$, for some $p, q\in \Spec(A)$. Since
$\Spec(A)$ is connected, we may also assume that $p\subseteq q$ and $\htt(q)=\htt(p)+1$.

Set $A'=\widetilde{A}_{\mathfrak {\widetilde{n}}}$ where
$\widetilde{A}$ is the integral closure of $A_{q}/pA_{q}$ in its field
of fractions and $\mathfrak {\widetilde{n}}$ is a maximal ideal of
$\widetilde{A}$. Then we have a commutative diagram of ring maps of the form:
$$
\xymatrix{
  A\ar[dr]_{\alpha} \ar[r]^{} & A' \ar[d]^{\rho}  \\
 & L
  }
$$
Moreover, the isomorphism
$$
\begin{aligned}
\Tor^{(B\sig)\ro}_{i} & \left((M\sig)\ro,(N\sig)\ro\right)
  \\
&  \cong  \Tor^{B\alf}_{i}(M\alf,N\alf)
\end{aligned}
$$
shows that 
$\widetilde{\theta}_{\alpha}(M,N)=\widetilde{\theta}_{\rho}(M\sig,N\sig)$.
Since the properties of being flat and essentially smooth are preserved under base
change, $A'$, $C\sig$,  $B\sig$, $M \sig$ and $N \sig$ satisfy the hypotheses of the Theorem.
Without loss of generality we may therefore assume $A$ is a dvr  with unique maximal ideal $\mathfrak{n}$, $L$ is the field of fractions
of $A$ and $E$ is the residue field $A/\mathfrak{n}$.

Set $\mathfrak n=(\pi)$ and fix a sufficiently large integer $j$.
We claim 
$$
\Supp_{B}(\Tor^{B}_{j}(M,N)) \subseteq V_{B}(I).
$$ 
Indeed, if $q' \notin V_B(I)$, then 
$B_{q'}$ is essentially smooth over $A$.
If $A \to B_{q'}$ is a local map, then
by \cite[23.7]{Mat}  $B_{q'}$ is regular. Otherwise, $\pi$ is
a unit in $B_{q'}$ and hence
$$
B_{q'}\otimes_{A}A\left[\frac{1}{\pi}\right]=
B_{q'}\left[\frac{1}{\pi}\right]=B_{q'}
$$ 
is regular. Either way, 
$\Tor^{B}_{j}(M,N)_{q'} = 0$.

It follows that $I^{l} \cdot \Tor^{B}_{j}(M,N)=0$ for some $l>0$. Using
the fact that $B/I$ is a finitely generated $A$-module, we see that
$\Tor^{B}_{n}(M,N)$ is a finitely generated $A$-module for all $n\gg
0$. Therefore, for all $n\gg 0$, we can write
\begin{equation} \label{E1}
\Tor^{B}_{n}(M,N)=A^{(r_{n})} \oplus G_{n}
\end{equation}
where $r_{n}$ is some integer and $\displaystyle{G_{n}}$ is the
torsion part of $\Tor^{B}_{n}(M,N)$. Let
$\displaystyle{t_{n}=\dim_{E}\left(\frac{G_{n}}{\pi
      G_{n}}\right)}$. Since $N$ is $A$-flat, we have a 
short exact sequence of complexes
\begin{equation} \label{E2}
0\rightarrow P_{\bullet}\otimes_{B}N
\stackrel{\pi}{\rightarrow} P_{\bullet}\otimes_{B}N \rightarrow
\left(P_{\bullet}\otimes_{B}N\right)\bet  \rightarrow 0.
\end{equation}
where $P_{\bullet} \stackrel{\sim}{\longrightarrow}M$ is a free
resolution of $M$ over $B$. Since $B$ is flat over $A$, $P_{\bullet}$
is also a flat $A$-resolution of $M$. Furthermore, as $M$ is $A$-flat,
$\H_{n}(P_{\bullet}\otimes_{A}E)=0$ for all $n\geq 1$.  Therefore
$P_{\bullet}\alf \stackrel{\sim}{\longrightarrow}M\alf$ is a free resolution of $M\alf$ over $B\alf$, and hence
\begin{equation} \label{E3}
\Tor^{B\alf}_{n}(M\alf,N\alf) \cong
\H_{n} \left( (P_{\bullet}\otimes_{B}N)\alf\right).
\end{equation}
Similarly, 
\begin{equation} \label{E4}
\Tor^{B\bet}_{n}(M\bet,N\bet) \cong
\H_{n} \left( (P_{\bullet}\otimes_{B}N)\bet \right).
\end{equation}
Now \eqref{E2} and \eqref{E4} yield the exact sequence
\begin{equation} \label{E5}
\Tor_{i}^{B}(M,N)\stackrel{\pi}{\rightarrow}\Tor_{i}^{B}(M,N)\rightarrow
\Tor^{B\bet}_{i}(M\bet,N\bet) \rightarrow \Tor_{i-1}^{B}(M,N).
\end{equation}
Notice \eqref{E3}  implies
$\displaystyle{\Tor^{B\alf}_{n}(M\alf,N\alf)\cong
\Tor^{B}_{n}(M,N)\alf}$ since $L$ is $A$-flat. Thus, by
\eqref{E1},  $\displaystyle{\Tor^{B\alf}_{j}(M\alf,N\alf)\cong
L^{(r_{j})}}$ and hence
$\displaystyle{\widetilde{\theta}_{\alpha}(M,N)=r_{2j}-r_{2j-1}}$.
We have, by \eqref{E5},   the short exact sequences
\begin{equation}  \label{E6}
0\rightarrow C_{i} \rightarrow \Tor^{B\bet}_{i}(M\bet,N\bet)
\rightarrow K_{i-1} \rightarrow 0
\end{equation}
where $K_{i}$ and $C_{i}$ are the kernel and cokernel of the multiplication by $\pi$ on $\Tor_{i}^{R}(M,N)$, respectively.
Now it follows from \eqref{E1} and \eqref{E6} that
$$
\widetilde{\theta}_{\beta}(M,N)=(r_{2j}+t_{2j}+t_{2j-1})-(r_{2j-1}+t_{2j-1}+t_{2j-2}).
$$
Since $\Tor^{B}_{n}(M,N)\cong \Tor^{B}_{n+2}(M,N)$ for all $n\gg 0$, we conclude that
$$
\widetilde{\theta}_{\alpha}(M,N)=r_{2j}-r_{2j-1}=\widetilde{\theta}_{\beta}(M,N).
$$
\end{proof}

\section{Some Consequences of the Main Theorem}

We explain how the main theorem implies that the $\theta$ pairing
factors through ``algebraic equivalence''. For this, it is helpful to
record the following simplified version of the theorem. 

\begin{cor} \label{C1}
Assume $k$ is a field, $S$ is essentially smooth over $k$, $f $ is an
element of $S$ such that the singular locus, $\Sing(R)$,  of $R := S/f$ is a finite
set of maximal ideals,
and that 
 $R/\fm \cong k$ for every $\fm \in \Sing(R)$ in the singular locus. (The latter
 condition holds, for
  example, if $k$ is algebraically closed and $R$ is
  finitely generated as a $k$-algebra). 
 If $A$ is a Noetherian $k$-algebra such that $\Spec(A)$ is connected
 and $M$ and $N$ are
  finitely generated $R\otimes_{k}A$-modules that are flat as
  $A$-modules, then 
$$
    \theta^{R\otimes_{k}k(q)}\left(M\otimes_{A}k(q), N\otimes_{A}
      k(q)\right)
$$ 
is independent of $q\in \Spec(A)$.
\end{cor}

\begin{proof} This follows by applying the theorem to the flat family
  of hypersurfaces with isolated  singularities given by $A$, $C = S
  \otimes_k A$, and $f \otimes 1 \in C$.
The fact that $R/\fm \cong k$ for all $\fm \in \Sing(R)$ implies that
$$
\widetilde{\theta}_\alpha(M,N)
=  \theta^{R\otimes_{k} L}
\left(M\otimes_{A} L, N\otimes_{A}
      L\right)
$$
where $\alpha: A \to L$ is any ring map with $L$ a field.
\end{proof}

\begin{defn}
For a finitely generated $k$-algebra $R$ with $k$ algebraically
closed, we define {\em algebraic equivalence} on
 $G(R)$ as the equivalence relation $\thicksim$  
generated by the following
 elementary relation: Given classes $\alpha$ and $\beta$ in  $G(R)$,
 $\alpha \thicksim \beta$ if there exist a finitely generated, smooth $k$-algebra $A$ with $\Spec(A)$ connected, maximal
ideals $\fm_{1}$ and $\fm_{2}$ of $A$ and a class $\gamma \in G(R\otimes_{k}A)$ such that $\alpha=i_{1}(\gamma)$ and
$\beta=i_{2}(\gamma)$, where
$$
i_{\epsilon}:G(R\otimes_{k}A) \rightarrow G(R)
$$
is the homomorphism defined on generators by
$$
[M]\mapsto \sum_{j\geq 0}(-1)^{j}\left[\Tor^{A}_{j}(M,A/\fm_{\epsilon})\right].
$$
Observe that 
$\Tor^{A}_{j}(M,A/\fm_{\epsilon})$ is 
an $R \otimes_k
A/\fm_\epsilon$-module and,
since $k \cong A/\fm_\epsilon$, we have that $R \cong R \otimes_k
A/\fm_\epsilon$.
\end{defn}

\begin{cor} \label{C2}
If $k$ an algebraically closed field, $S$ is a smooth $k$-algebra, and
$R = S/f$ is a hypersurface with isolated singularities, 
then the  $\theta$ pairing on $G(R)$  factors through algebraic equivalence; that is, there is a
commutative diagram of the form:
$$
\xymatrix{
G(R)^{\otimes 2}_{\QQ}\ar[r]^-\theta \ar@{->>}[d] & \QQ \\
\left(\frac{G(R)_{\QQ}}{\thicksim}\right) ^{\otimes 2} \ar[ru]&
}
$$
\end{cor}

\begin{proof} For a smooth $k$-algebra $A$ such that $\Spec(A)$ is
  connected, a pair of maximal
  ideals $\fm_1, \fm_2$ of $A$, and a finitely generated $R \otimes_k A$-module
$M$ that is $A$-flat, we 
apply Corollary \ref{C1} to $M$ and $N = T \otimes_k A$, where $T$ is a finitely
generated $R$-module. This gives $ \theta(M_1, T) = \theta(M_2, T)$, where
$M_\epsilon := M \otimes_A A/\fm_\epsilon$.
Since $G(R\otimes_{k}A)$ is generated by classes of modules that are
flat over $A$, it follows that $\theta(-,T)$ annihilates the image of 
$i_1 - i_2: G(R \otimes_k A) \to G(R)$. 
\end{proof}

\begin{prop} \label{Dao} (H. Dao, cf. \cite[2.8]{Da1}) Let $R=S/(f)$
  and let $M$ and $N$ be finitely generated $R$-modules. Assume
  $\Sing(R)=\{
  \mathfrak{\overline{n}}_{1},\mathfrak{\overline{n}}_{2}, \dots,
  \mathfrak{\overline{n}}_{r}\}$ where 
$\mathfrak{n}_{1}, \dots, \mathfrak{n}_{r}$ 
are maximal
  ideals of $S$ such that 
  $S_{\mathfrak{n}_{i}}$ is an unramified regular local ring for each
  $i$. If $\theta^{R}(M,N)=0$, then $(M,N)$ is $\Tor$-rigid over $R$.
\end{prop}

\begin{proof} Suppose that $\Tor_{n}^{R}(M,N)=0$ for some $n\geq 0$. 
Since rigidity holds over regular local rings
(\cite[Corollary 2.2]{Au} and
\cite[Corollary 1]{Li}), $\Supp(\Tor_{i}^{R}(M,N))
\subseteq \Sing(R)$ for all $i\geq n$. Hence
$\len(\Tor_{i}^{R}(M,N))<\infty$ for all $i\geq n$. As
$\theta^{R}(M,N)=\theta^{T^{-1}R}(T^{-1}M,T^{-1}N)$ for
$T=R-\bigcup_{i}\mathfrak{\overline{n}}_{i}$, we may assume $S$ is
semi-local with maximal ideals $\mathfrak{n}_{1}$,
$\mathfrak{n}_{2}$, \dots, $\mathfrak{n}_{r}$. Then
$$
\chi^{S}_{n}(M,N)=\chi^{S_{\mathfrak{n}_{1}}}_{n}(M_{\mathfrak{n}_{1}},N_{\mathfrak{n}_{1}})+\dots
+\chi^{S_{\mathfrak{n}_{r}}}_{n}(M_{\mathfrak{n}_{r}},N_{\mathfrak{n}_{r}})
$$ 
where $\ds{\chi^{S}_{n}(M,N)=\sum_{j \geq
n}(-1)^{j-n}\cdot \len(\Tor^{R}_{j}(M,N))}$ is the higher Euler
characteristic of the pair $(M,N)$. Write
$$
\theta^{R}(M,N)=\len(\Tor_{2e}^{R}(M,N))-\len(\Tor_{2e-1}^{R}(M,N))
$$
for some $e\gg 0$ and consider the following exact sequence
\cite{J}:
$$
\begin{aligned}
0 \to \Tor_{2e}^{R}(M,N) \to & \cdots \to \Tor_{n}^{R}(M,N)
\\
&  \to \Tor_{n+1}^{S}(M,N) \to \Tor_{n+1}^{R}(M,N) \to C \to
0
\end{aligned}
$$
Taking the alternating sum of the lengths, we deduce that:
$$\len(C)+\chi_{n+1}^{S}(M,N)=(-1)^{2e-n}\cdot \theta_{R}(M,N)+\len(\Tor_{n}^{R}(M,N))$$
It follows from \cite{Ho2} and \cite{Li} that, for all $t \geq n$,
$\chi^{S_{\mathfrak{n}_{i}}}_{t}(M_{\mathfrak{n}_{i}},N_{\mathfrak{n}_{i}})\geq
0$ and
$\ds{\chi^{S_{\mathfrak{n}_{i}}}_{t}(M_{\mathfrak{n}_{i}},N_{\mathfrak{n}_{i}})=0}$
if and only if $\ds{\Tor_{j}^{
S_{\mathfrak{n}_{i}}}(M_{\mathfrak{n}_{i}},N_{\mathfrak{n}_{i}})=0}$
for all $j\geq t$. Thus $\ds{C=0}$ and $\ds{\chi_{n+1}^{S}(M,N)=0}$.
This implies $\Tor_{n+1}^{R}(M,N)=0$ and hence $\Tor_{i}^{R}(M,N)=0$
for all $i\geq n$.
\end{proof}

Using Dao's result (as extended in Proposition \ref{Dao}), we may use
Theorem \ref{main} to give a statement about $\Tor$-rigidity:

\begin{cor} \label{cor2}
Assume $A, C, f, B = C/f$ form a flat family of hypersurfaces with isolated
singularities and $M, N$ are finitely generated $B$-modules that are
$A$-flat. Let $\alpha: A \to L$ and $\beta: A \to
  E$ be ring maps where $L$ and $E$ are fields. If
$\ds{\widetilde{\theta}_{\alpha}(M,N)=0}$, then the pair
$(M\bet,N\bet)$ is $\Tor$-rigid over $B\otimes_{A}E$. In particular, 
if $\pd_{B\otimes_{A}L}(M\otimes_{A}L)<\infty$, then the pair $(M\bet,N\bet)$ is $\Tor$-rigid over $B\otimes_{A}E$.
\end{cor}

\begin{proof} It follows from Theorem \ref{main} that $\ds{\widetilde{\theta}_{\gamma}(M,N)=0}$, where
$\gamma$ is the ring map from $A$ to the algebraic closure
$\overline{E}$ of $E$ induced by $\beta$. In view of Remark
\ref{rmk2} we have $\ds{\theta^{B\otimes_{A}\overline{E}}
(M\otimes_{A}\overline{E},N\otimes_{A}\overline{E})=0}$. Hence, by
Proposition \ref{Dao}, the pair
$(M\otimes_{A}\overline{E},N\otimes_{A}\overline{E})$ is
$\Tor$-rigid over $B\otimes_{A}\overline{E}$. The result now follows
from the fact that $B\otimes_{A}E \map{} B\otimes_{A}\overline{E}$
is faithfully flat. The final assertion holds since the $\theta$ pairing clearly vanishes
if either argument has finite projective dimension.
\end{proof}

It is worth mentioning that $\pd_{B\otimes_{A}L}(M\otimes_{A}L)<\infty$ does not imply $\pd_{B\otimes_{A}E}(M\otimes_{A}E)<\infty$ in general:

\begin{eg} Let $k$ be a field with $\chr(k) \neq 2$ and let $A=k[t]$, $C=k[x,y,t]$, $L=k(t)$, $E=A/(t)$, $f=y^{2}-x(x-t)$ and $M=B/(x,y)$. Then  $B\otimes_{A}L=k(t)[x,y]/(f)$ is regular since it is smooth over $k(t)$. Thus $\pd_{B\otimes_{A}L}(M\otimes_{A}L)<\infty$. However, as $B\otimes_{A}E \cong k[x,y]/(y^2-x^2)$ is singular at $(x,y)$, $M\otimes_{A}E \cong k$ is of infinite projective dimension over $B\otimes_{A}E$.
\end{eg}

\section{Applications to Matrix Factorizations}

In certain situations it is possible to lift a given hypersurface with
isolated singularities to a non-trivial flat family of such. For
example, starting with one in characteristic $p$, it is sometimes
possible to
lift it to a family that is  indexed by a mixed characteristic
ring. In some situations, this allows one to deduce results in the
characteristic $p$ setting from known results in the characteristic
$0$ setting.

\begin{prop} \label{prop1} Let $F$ be a field, $S=F[x_1, \dots, x_n]$
  and $f\in (x_1, \dots, x_n) \subset S$. Assume $(x_1, \dots, x_n)$
  is an isolated point in the locus of non-smooth points of $S/f$
--- i.e., assume 
there exists $g\in S \setminus (x_1, \dots, x_n)$ such
  that $\Spec\left((S/f)\left[\frac{1}{g}\right] \right)-\{(x_1,
  \dots, x_n)\}$ is smooth over $F$. Let $(A, \mathfrak{m}_{A})$ be a
  Henselian local ring (e.g., a complete local ring) with residue field $F$ and let
  $\widetilde{S}=A[x_1, \dots, x_n]$ and $\fn = \fm_A + (x_1, \dots, x_n)$.

For any $\widetilde{f}\in \widetilde{S}$
  such that $\widetilde{f}\equiv f \pmod{\mathfrak{m}_{A}}$, 
there exits $h\in
    \widetilde{S} \setminus \fn$ such that
$\left(\widetilde{S}/\widetilde{J}\right) \left[\frac{1}{h}\right]$
 is a finitely generated $A$-module, where
  $\widetilde{J}=\left(\frac{\partial \widetilde{f}}{\partial
        x_1},\dots, \frac{\partial \widetilde{f}}{\partial x_n},
      \widetilde{f}\right)$ and $\frac{\partial -}{\partial x_i}$ are
    the evident $A$-linear derivations.
\end{prop}

\begin{proof} 
We may assume $\ds{\widetilde{J} \subseteq \mathfrak{n}}$, for otherwise
we may pick $h$ to be any element of $\widetilde{J}-\mathfrak{n}$.

Let $T = \widetilde{S}/\widetilde{J}$. Since $A \to T$ has finite type,
by the upper semi-continuity of fiber dimensions \cite[13.1.3]{EGAIV}, the set
of primes $q \in \Spec(T)$ such that $\dim((T/(q \cap A))_{q}= 0$ is an
open subset of $\Spec(T)$.
Since $(T/(\fn \cap A))_{\fn} \cong (S/f)_{(x_1, \dots, x_n)}$ and 
$(x_1, \dots, x_n)$ is an isolated singularity of $S/f$,
we get that $\fn$ belongs to this open subset.
It follows that there exists an $\alpha \in T \setminus \fn$ such that
$A \to V := T\left[\frac{1}{\alpha}\right]$ is
  quasi-finite.

Since $A$ is a Henselian local ring, it follows from \cite[I.4.2]{Milne} that $V=V_{0} \times V_{1} \times
\cdots \times V_{m}$ where $V_{i}$ is a finitely generated $A$-module
for all $i=1, \dots, m$ and $\fp \cap A =\{0\}$ for all $\fp\in
\Spec(V_{0})$. Identifying $\Spec(V)$  with the disjoint union of $\Spec(V_{i})$, we see that $\mathfrak{n} \notin \Spec(V_{0})$. 
Without loss of generality, assume $\mathfrak{n}\in
\Spec(V_{1})$, so that $(V_{i})_{\mathfrak{n}}=0$ for $i\neq 1$. Then
there exits $\beta \in V \setminus \mathfrak{n}$ such that
$V_{1}\left[\frac{1}{\beta}\right]=V_{1}$ and
$V_{i}\left[\frac{1}{\beta}\right]=0$ for $i \ne 1$. 

 Now
set $h=\alpha \cdot \beta$. Then
$(\widetilde{S}/\widetilde{J})\left[\frac{1}{h}\right] \cong
V\left[\frac{1}{\beta}\right] \cong V_{1}$ is
a finitely generated $A$-module.
\end{proof}

In the homogeneous case, the previous result may be refined:

\begin{prop} \label{prop2} Let $F$ be a field, $S=F[x_1, \dots, x_n]$
  and $f\in S$ be a homogeneous element such that 
  $(x_1, \dots, x_n)$ is the only non-smooth 
point of $S/f$. Let $A$ be a complete
  local ring with residue field $F$. Let $C=A[x_1, \dots,
  x_n]$ and let $\widetilde{f}\in C$ be any homogeneous element
  such that $\widetilde{f}\equiv f \pmod{\fm_{A}}$. Set $B=C/(\widetilde{f})$ and $I=\left(\frac{\partial
        \widetilde{f}}{\partial x_1}, \dots, \frac{\partial
        \widetilde{f}}{\partial x_n} \right) \subset B$. Then $B/I$ is a finitely generated $A$-module.
\end{prop}

\begin{proof} Observe that 
$B/I$ is a graded $C$-module. 
By Proposition \ref{prop1}, there exists an $h \notin  
  \mathfrak{m}_{A}+ (x_1, \dots, x_n)$ such that
  $(B/I)\left[\frac{1}{h}\right]$ is a finitely generated $A$-module. Consider a filtration 
$$
0=M_{0}\subseteq  M_{1} \subseteq \dots \subseteq M_{n}=B/I
$$
by graded submodules, such that $\ds{M_{i}/M_{i-1} \cong C/q_{i}}$
where each $q_{i}$ is a homogeneous prime ideal.  Then $q_{i}
\subseteq \mathfrak{m}_{A} +  (x_1, \dots, x_n)$ and hence $h \notin q_{i}$ for each $i$. Since
$M_{i}/M_{i-1} \hookrightarrow
(M_{i}/M_{i-1})\left[\frac{1}{h}\right]$ is injective and
$(B/I)\left[\frac{1}{h}\right]$ is module finite over $A$, it follows that
the modules $M_i/M_{i-1}$, and hence $B/I$, are  finitely generated $A$-modules.
\end{proof}

\begin{defn} (\cite{Ei}) For a ring $V$ and homogeneous element $f \in 
V[x_0,  \dots, x_n]$ (equipped with the standard grading), a {\em homogeneous matrix factorization} of $f$ is
  pair of $m \times m$ matrices $(A,B)$ with entries in 
$V[x_0,  \dots, x_n]$ such that each determines a map of graded free
modules and $AB = fI_m = BA$. 
\end{defn}

\begin{rmk}
The condition that a $p \times q$ matrix $A = (a_{i,j})$ with entries
in $S := V[x_0,
\dots, x_n]$ determines a map of graded free modules is equivalent to
the condition that every entry of $A$ is homogeneous and there are
integers $(e_1, \dots, e_q), (d_1, \dots d_p)$ so that $|a_{i,j}| =
e_j - d_i$ for all $i,j$. For given such integers, the map
$$
A : \bigoplus_{j=1}^q S(-e_j) \to \bigoplus_{i=1}^p S(-d_i)
$$
is a map of graded free modules, i.e., homogeneous of degree $0$.  
\end{rmk}

For a homogeneous matrix factorization $(A,B)$ of $f \in S := F[x_0,
\dots, x_n]$ with $F$ a field, the cokernel of $A$ (viewed as a map of graded free modules)
is a graded maximal Cohen-Macaulay module (cf. \cite{BH}) over $R := F[x_0, \dots,
x_n]/(f)$. Conversely, if $M$ is a graded maximal Cohen-Macaulay 
$R$-module, then picking a graded free resolution of it as an $S$-module
leads to a homogeneous matrix factorization.

As mentioned in the introduction, the theta pairing is known to vanish
for hypersurfaces with isolated singularities of even dimension that
are of the form $R = F[x_0, \dots,
x_n](f)$ with $f$ homogeneous and $F$ any field \cite{MPSW} or $R = \CC[[x_0, \dots,
x_n]]/(g)$ \cite{PV}. Moreover, in the homogeneous case with 
$n = \dim(R)$  odd and $\ch(F) = 0$, the pairing
$(-1)^{\frac{n+1}{2}} \theta$ is semi-positive definite.
The following result indicates a method of
translating this fact to the positive characteristic setting in
favorable situations.

\begin{cor} \label{C3}
Let $F$ be an arbitrary field and let
  $S=F[x_0, \dots, x_n]$ be given the standard grading where $n$ is odd. Let $(A,B)$ be a homogeneous matrix
  factorization of a homogeneous element $f\in S$ such that $(x_0, \dots, x_n)$ is the only
  non-smooth prime of $R = S/f$.  

Suppose there exist a complete dvr $(V,\fm_V)$ with residue field
$F$ and field of fractions of characteristic $0$, a homogeneous element $\tilde{f} \in V[x_0,
\dots, x_n]$ and a homogeneous matrix
factorization $(\tilde{A}, \tilde{B})$ of $\tilde{f}$ such that 
$\tilde{f} \equiv f$, $\tilde{A} \equiv A$ and
$\tilde{B} \equiv B$ modulo $\fm_V$. 
Then  $\theta(M,M)\geq 0$ if $n \equiv 3
  \pmod{4}$ and $\theta(M,M)\leq 0 $ if $n \equiv 1 \pmod{4}$.
\end{cor}

\begin{proof} 
Let $\tilde{M}$ be the graded $V[x_0, \dots, x_n]$-module given by the
cokernel of $\tilde{A}$. Then by Proposition \ref{prop2}, 
$V$, $V[x_0, \dots, x_n]$, and $\tilde{f}$ from a flat family of hypersurfaces
with isolated singularities. 
Moreover, $\tilde{M}$ is a flat $V$-module
and 
hence by Theorem \ref{main} we have
$\tilde{\theta}_\alpha(\tilde{M}, \tilde{M}) = 
\tilde{\theta}_\beta(\tilde{M}, \tilde{M})$ where $\alpha: V \to F$
is the map to the residue field of $B$ and $\beta: V \to L$ is the map
to its field of fractions. Since the only singular prime of $R$,
respectively $R' := L[x_0, \dots, x_n]/\tilde{f}$, has residue field $F$, respectively $L$,
we have
$\tilde{\theta}_\alpha(\tilde{M}, \tilde{M}) = \theta^R(M,N)$ and
$\tilde{\theta}_\beta(\tilde{M}, \tilde{M}) = \theta^{R'}(M,N)$. The
result now follows from \cite[Theorem 3.4]{MPSW}.
 \end{proof}

\begin{eg} \label{example}
Let $F$ be a field of characteristic $p>0$ and let
  $f \in S=F[x_0, \dots, x_n]$ be a homogeneous element such that
  $(x_0, \dots, x_n)$ is the only 
non-smooth prime of 
  $S/f$. If $(A,B)$ is a homogeneous matrix
  factorization for $f$ satisfying the additional constraint that
$\det(A)=f$, then a lifting $(\tilde{A}, \tilde{B})$ as in Corollary 
\ref{C3} always exists. Namely, pick any complete dvr of mixed characteristic $V$ with residue
field $F$ (e.g., the ring of Witt vectors \cite{Bour2}) and let
$\widetilde{A}$ be any matrix
  with coefficients in $V[x_0, \dots, x_n]$ such that $\widetilde{A} \equiv
  A \pmod{\mathfrak{m}_{V}}$. 
Then set 
  $\widetilde{B}=\widetilde{A}^{\textnormal{adj}}$, and we get 
  $\widetilde{A} \cdot \widetilde{B} =\widetilde{f}\cdot I_{r}$ with
$\widetilde{f} \equiv f  \pmod{\mathfrak{m}_{V}}$.
Thus, for such matrix factorizations, we have 
 $\theta(M,M)\geq 0$ if $n \equiv 3
  \pmod{4}$, and $\theta(M,M)\leq 0 $ if $n \equiv 1 \pmod{4}$, where
  $M = \coker(A)$.
\end{eg}

Assume $R$ is a $d$-dimensional local hypersurface that is an isolated singularity.
If $d\geq 4$, then there are no non-free maximal Cohen-Macaulay $R$-modules of rank one. 
(We are grateful to Hailong Dao and Roger Wiegand for explaining this
to us.) For by a theorem of Grothendieck \cite{SGA},
if $R_{p}$ is factorial for all $p\in \Spec(R)$ such that
$\height(p)\leq 3$, then $R$ is factorial. Since a maximal
Cohen-Macaulay module $M$ of rank one 
determines an element in the divisor class group of $R$, $M$ must be
free.  

On the other hand, rank one maximal Cohen-Macaulay modules over
three   dimensional hypersurface singularities exist. For example,  the ideal $(x,y)$
of $R=k[[x, y, z, w]]/(xw-yz)$ is a non-free rank one maximal Cohen-Macaulay module.                                                     
In case $d=3$, Dao and the second author of this paper 
proved that $\theta^{R}(M,M)\geq 0$ for any finitely generated
$R$-module $M$. (See also \cite{Da3}.)

The condition $\det(A) = f$ is equivalent to the condition that $M$ is a rank one maximal Cohen-Macaulay module (cf., for example, \cite[2.2.5]{Hov}).
Therefore the conclusion in Example \ref{example} also follows from the discussion above.

\providecommand{\bysame}{\leavevmode\hbox to3em{\hrulefill}\thinspace}
\providecommand{\MR}{\relax\ifhmode\unskip\space\fi MR }
\providecommand{\MRhref}[2]{%
  \href{http://www.ams.org/mathscinet-getitem?mr=#1}{#2}
}
\providecommand{\href}[2]{#2}

\bibliography{a}
\bibliographystyle{plain}

\end{document}